\RequirePackage{fix-cm}
\documentclass[smallextended,numbook]{svjour3}        
\smartqed  
\usepackage{graphicx}

\usepackage{latexsym}
\usepackage{amssymb}
\usepackage{amsmath}
\usepackage{relsize}
\usepackage{cite}
\usepackage{hyperref}  
\usepackage{url}
\usepackage{graphicx}
\usepackage{subcaption}
\usepackage{xcolor}
\usepackage{hyperref}
\usepackage{amsfonts}
\usepackage{appendix}

\newcommand{\nc}{\newcommand}
\newcommand{\argmin}{\operatornamewithlimits{argmin}}
\nc{\nt}{\newtheorem}
\nt{thm}{Theorem}[section]
\nt{cor}[thm]{Corollary}
\nt{prop}[thm]{Proposition}
\nt{obs}[thm]{Observation}
\nt{lem}[thm]{Lemma}
\nt{defn}[thm]{Definition}
\nt{exa}[thm]{Example}
\nt{rem}[thm]{Remark}
\nt{ass}[thm]{Assumption}
\nt{alg}[thm]{Algorithm}
\nt{con}[thm]{Conjecture}

\def\ep{\varepsilon}
\def\B{{\cal B}}
\nc{\ip}[2]{\mbox{$\langle #1,#2 \rangle$}}
\nc{\pf}{\noindent{\bf Proof\ \ }}
\nc{\finpf}{\hfill{$\Box$}\linespace}
\nc{\linespace}{\vspace{\baselineskip} \noindent}
\nc{\R}{{\bf R}}
\nc{\A}{\mathcal{A}}
\nc{\E}{{\bf E}}
\nc{\tto}{\rightrightarrows}
\nc{\cl}{\mbox{\rm cl}\,}
\nc{\cls}{ \mbox{{\scriptsize {\rm cl}}}\,}
\nc{\conv}{\mbox{\rm conv}\,}
\nc{\aff}{\mbox{\rm aff}\,}
\nc{\rb}{\mbox{\rm rb}\,}
\nc{\ri}{\mbox{\rm ri}\,}
\nc{\inter}{\mbox{\rm int}\,}
\nc{\kernel}{\mbox{\rm ker}\,}
\nc{\bd}{\mbox{\rm bd}\,}
\nc{\sign}{\mbox{\rm sign}\,}
\nc{\Diag}{\mbox{\rm Diag}\,}
\nc{\rint}{\mbox{\rm rint}\,}
\nc{\epi}{\mbox{\rm epi}\,}
\nc{\supp}{\mbox{\rm supp}\,}
\nc{\gph}{\mbox{\rm gph}\,}
\nc{\rge}{\mbox{\rm rge}\,}
\nc{\rgel}{\mbox{\rm {\scriptsize rge}}\,}
\nc{\sepi}{\mbox{\rm {\scriptsize epi}}\,}
\nc{\sbd}{\mbox{\rm {\scriptsize bd}}\,}
\nc{\dom}{\mbox{\rm dom}\,}
\nc{\lin}{\mbox{\rm lin}\,}
\nc{\detr}{\mbox{\rm det}\,}
\nc{\para}{\mbox{\rm par}\,}
\nc{\crit}{\mbox{\rm crit}\,}
\nc{\spann}{\mbox{\rm span}\,}
\nc{\cone}{\mbox{\rm cone}\,}
\nc{\diag}{\mbox{\rm Diag}\,}
\nc{\fix}{\mbox{\rm Fix}}
\nc{\rank}{\mbox{\rm rank}\,}
\nc{\tr}{\mbox{\rm tr}\,}
\nc{\vect}{\mbox{\rm vec}\,}

\newcommand{\argmax}{\operatornamewithlimits{argmax}}

\begin{document}

\title{Extreme point inequalities and geometry of the rank sparsity ball
}


\author{D. Drusvyatskiy         \and
        S.A. Vavasis \and H. Wolkowicz
}


\institute{
D. Drusvyatskiy \\ 
Department of Mathematics, University of Washington, Seattle, WA 98195-4350;\\
              \email{ddrusvya@uwaterloo.ca};
               URL: \url{people.orie.cornell.edu/dd379}.          
            \\
           S.A. Vavasis \\
\email{vavasis@math.uwaterloo.ca};
URL: \url{www.math.uwaterloo.ca/\~vavasis}.
	  \\
	H. Wolkowicz \\
\email{hwolkowicz@uwaterloo.ca};
URL: \url{orion.uwaterloo.ca/\~hwolkowi}.
}

\date{Received: date / Accepted: date \footnote{Department of Combinatorics and Optimization, University of Waterloo, Waterloo, Ontario, Canada N2L 3G1}}

\maketitle

\begin{abstract}
We investigate geometric features of the unit ball corresponding to the sum of the nuclear norm of a matrix and the $l_1$ norm of its entries --- a common penalty function encouraging joint low rank and high sparsity. As a byproduct of this effort, we develop a calculus (or algebra) of faces for general convex functions, yielding a simple and unified approach for deriving inequalities balancing the various  features of the optimization problem at hand, 
at the extreme points of the solution set.

\end{abstract}

\section{Introduction}
Recovery of a structured signal from a small number of linear measurements has been a hot topic of research in recent years. Notable examples include recovery of sparse vectors \cite{CT, D, CMT}, low-rank matrices \cite{RFP, CR}, and a sum of sparse and low-rank matrices \cite{CLYW,CPW}, to name a few. An overarching theme in this area is to replace a difficult nonconvex objective by a convex surrogate, which usually arises as the convex envelope of the objective on a neighborhood of the origin. For example, one may replace the rank of a vector $x$ by the $l_1$-norm $\|x\|_1$ and the rank of a matrix $X$ by the nuclear norm $\|X\|_*$. In practice, however, it is often the case that the signal that we are attempting to recover is simultaneously structured. In this case, it is common practice to simply use the sum of the convex surrogates to enforce the joint structure. We note in passing that from a compressed sensing point of view, this strategy may be lacking. Oymak et al. \cite{Fazel_neg}  argue that sums of norms (or more general combinations) do not appear to give results stronger than individual norms; we return to
this point in Section~\ref{sec:disc}. Nevertheless, this  is effective and common in practice.

To ground the discussion, suppose that we are interested in finding a matrix satisfying a linear system that simultaneously has low rank and is sparse. This situation arises in a vast number of applications. See for example sparse phase retrieval \cite{PRO,Phase} and cluster detection \cite{sim_sparse, AV}, and references therein. As alluded to above, it is popular to then consider the joint norm  
$$\|X\|_{1,*}:=\|X\|_1 +\theta\|X\|_{*},$$
where $\|X\|_1$ is the $l_1$-norm of the entries of $X$, and
the parameter $\theta >0$ balances the trade off between sparsity and rank. A proximal-point based algorithm for optimizing this norm on an affine subspace has been proposed in \cite{prox_norm}. In contrast to previous research on recovery of jointly structured models, our focus is not set in the context of compressed sensing. Rather we begin by asking a more basic convex analytic question: 
\begin{center}
How does the facial structure of each norm $\|\cdot\|_1$ and $\|\cdot\|_*$ individually influence the facial structure of the unit ball $\B_{1,*}:=\{X:\|X\|_{1,*}\leq 1\}$? 
\end{center}
To adequately address this question it seems that one needs to investigate the trade-off between rank and sparsity --- a topic that to the best of our knowledge has not been explored nearly enough. We hope that this short note will at least begin to rectify this discrepancy. For the sake of readers' intuition, the unit balls corresponding to the three norms mentioned above, restricted to $2\times 2$ symmetric matrices, are illustrated below. 
\begin{figure}[h]
        \centering
        \begin{subfigure}[b]{0.3\textwidth}
                \centering
                \includegraphics[width=\textwidth]{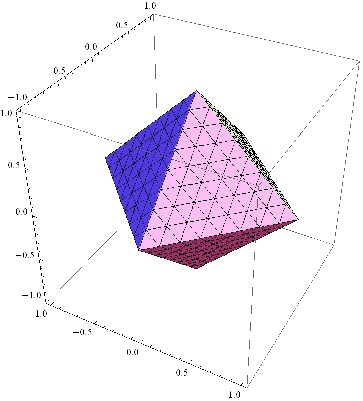}
                \caption{$\|\cdot\|_{1}$-ball}
        \end{subfigure}%
        ~ 
        \begin{subfigure}[b]{0.3\textwidth}
                \centering
                \includegraphics[width=\textwidth]{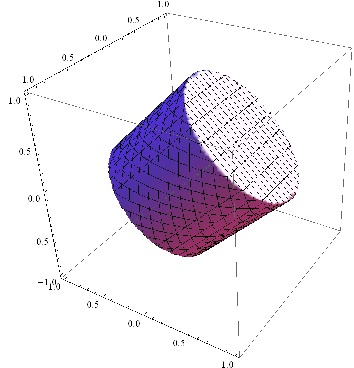}
                \caption{$\|\cdot\|_{*}$-ball}
        \end{subfigure}
        ~ 
        \begin{subfigure}[b]{0.3\textwidth}
                \centering
                \includegraphics[width=\textwidth]{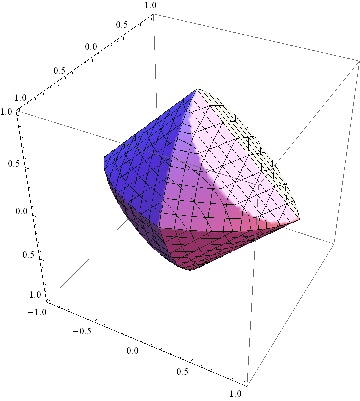}
                \caption{$\|\cdot\|_{1,*}$-ball}
        \end{subfigure}
\end{figure}

To summarize the main results, we will show that any extreme point $X$ of the ball $\B_{1,*}$ satisfies the inequality
\begin{equation}\label{eqn:rank_sparse}
\frac{r(r+1)}{2}- |I|\leq 1,
\end{equation}
where $r$ and $|I|$ are the rank 
and the number of zero entries of $X$, respectively.
Moreover, surprisingly, we will see that all the vertices of the ball 
$\B_{1,*}$  --- points where the normal cone has nonempty interior --- are simply the extreme points of  $\{X: \|X\|_1\leq \frac{1}{2}\}$, that is no ``new'' vertices are created when $\|\cdot\|_1$ and $\|\cdot\|_*$ are summed. The latter depends on an interesting observation made precise in Theorem~\ref{thm:vert}: the set of matrices with a prescribed rank and sparsity pattern is rarely small; such a set contains a naturally occurring smooth submanifold whose size depends only on the connectivity of the adjacency graph. 

These results, in turn, have immediate implications on problems of low rank sparse recovery. Namely, when minimizing the norm $\|X\|_1 +\theta \|X\|_{*}$ over matrices satisfying a linear system $\mathcal{A}(X)=b$, any extreme point $X$ of the solution set
satisfies the inequality
\begin{equation}\label{eqn:rank_sparse_opt}
\frac{r(r+1)}{2}-|I| \leq 1 +d,
\end{equation}
where $d$ is the dimension of the range of $\mathcal{A}$ (i.e. the number of linear measurements). Moreover we prove that the problem of minimizing a linear functional $\langle V, X\rangle$ subject to $\Vert X\Vert_1 + \theta \Vert X\Vert_*\le 1$ will
recover a sparse rank-one matrix for a positive measure subset of matrices $V$, a key result for the work of Doan and Vavasis \cite{DV13} and Doan, Toh and Vavasis \cite{prox_norm}, who use the joint norm $\| \cdot\|_{1,*}$ to find hidden rank-one blocks inside large matrices. 

Those well-versed in the theory of semi-definite programming will see that equations (\ref{eqn:rank_sparse}) and (\ref{eqn:rank_sparse_opt}) are reminiscent of the foundational results of \cite{P_ext,B_ext}, where the authors derive bounds on the rank of extreme points of the feasible regions of SDP's in terms of the number of constraints, and the more general theory for conic linear programs \cite{handbook}. 
The basic ingredient for such results is a theorem of Dubins \cite[Page 116]{SW} stating that a set is a face of an intersection of two convex sets if and only if it is an intersection of two faces. In the current manuscript, we take this idea further by developing a calculus (or algebra) of faces for general convex functions.
Indeed, one of the major successes of modern convex analysis is that sets and functions are put on an equal footing. Our analysis of the facial structure of the rank sparsity ball nicely illustrates how such a calculus can generally yield a simple and transparent way of obtaining inequalities (analogous to (\ref{eqn:rank_sparse}) and (\ref{eqn:rank_sparse_opt})) balancing the various  features of the optimization problem at hand, 
at the extreme points of the solution set. In particular, this technique easily adapts to the sum of many other ``basic'' norms --- a common feature of joint structure recovery.

The outline of the manuscript is as follows. In Section~\ref{sec:face}, we recall some basic tools of convex analysis and record a facial calculus. In Section~\ref{sec:rank_sp}, we study the extreme points and the vertices of the rank sparsity ball, in part using results of the previous section. In Section~\ref{sec:rec}, we prove that the vectors exposing rank one matrices with a fixed sparsity pattern  have nonzero measure.

\section{Faces of convex functions}\label{sec:face}
\subsection{Notation and preliminaries}
We begin by establishing some notation and recalling basic tools of convex analysis. We will in large part follow the notation of \cite{con_ter}.
Throughout, the symbol $\E$ will denote a Euclidean space (finite-dimensional real inner product space) with norm $\|\cdot\|$ and inner-product $\langle \cdot,\cdot \rangle$. The closed ball of radius $\ep >0$ around a point $\bar{x}$ will be denoted by $\B_{\ep}(\bar{x})$, while the closed unit ball will be denoted by ${\bf\B}$. The interior, boundary, and closure of a set $Q\subset\E$ will be written as $\inter Q$, $\bd Q$, and $\cl Q$, respectively.
The linear span, affine span, convex hull, (nonconvex) conical hull, and convex conic hull of $Q$   will be written as $\spann Q$, $\aff Q$, $\conv Q$, $\R_+ Q$, and $\cone Q$, respectively. The interior and boundary of $Q$ relative to its affine span will be denoted by $\ri Q$ and $\rb Q$, respectively.
We will consider functions $f$ on $\E$ taking values in the extended real line $\overline{\R}:=\R\cup\{\pm\infty\}$. We will always assume that such functions are {\em proper}, meaning  they never take the value $-\infty$ and are not identically $+\infty$.
For a function $f\colon\E\to\overline{\R}$, we define the {\em domain},  {\em gph}, and {\em epigraph} of $f$, respectively, to be
\begin{align*}
\dom f &= \{x\in\E: f(x) <\infty\},\\
\gph f &= \{(x,f(x))\in \E\times\R : x\in\dom f\},\\
\epi f &= \{(x,r)\in \E\times\R : f(x)\leq r\}.
\end{align*}
If in addition $Q$ is a subset of $\E$, then the symbol $\gph f\big|_Q$ will simply stand for $(Q\times \R)\cap \gph f$.
The symbol $[f\leq r]$ will denote the sublevel set $\{x\in \E : f(x)\leq r\}$. Analogous notation will be reserved for $[f=r]$.
A function $f\colon\E\to\overline{\R}$ is {\em lower-semicontinuous} (or {\em lsc} for short) if the epigraph $\epi f$ is closed. The {\em subdifferential} of a convex function $f$ at $\bar{x}$ is the set
$$\partial f(\bar{x}):=\{v\in\E: f(x)\geq f(\bar{x})+\langle v,x-\bar{x}\rangle \textrm{ for all }x\in \E\}.$$
The {\em indicator function} of a convex set $Q$, written $\delta_Q$, is defined to be zero on $Q$ and $+\infty$ elsewhere. The {\em normal cone} to $Q$ at a point $\bar{x}\in Q$ is $N_Q(\bar{x}):=\partial \delta_{Q}(\bar{x})$ while the {\em tangent cone} is the polar $T_Q(\bar{x}):=(N_Q(\bar{x}))^{o}$, where for any convex cone $K$ we define $K^{o}:=\{v:\langle x,v\rangle\leq 0\}$. 

With any function $f\colon\E\to\overline{\R}$, we associate the {\em Fenchel conjugate} $f^{*}\colon\E\to\overline{\R}$ by setting $$f^{*}(u):= \sup_{x\in\E}\, \{\langle u,x \rangle - f(x)\} .$$
Whenever $f$ is lsc and convex,  we have $(f^*)^{*}=f$ and $\partial f^{*}= (\partial f)^{-1}$, where we use the convention $(\partial f)^{-1}(u):=\{x: u\in\partial f(x)\}$.
In particular, when $K$ is a closed convex cone, the equations $\delta_K^{*}=\delta_{K^{o}}$ and $N_{K^{o}}=(N_K)^{-1}$ hold.

\subsection{Faces of functions}\label{subsec:face_gen}
Consider a convex set $Q\subset\E$. Classically, a {\em face} of $Q$ is a convex subset $F$ of $Q$ such that every closed segment in $Q$ whose relative interior intersects $F$ must lie fully in $F$. A face $F$ of $Q$ is a {\em minimal face} at $\bar{x}$ if for any other face $F'$ containing $\bar{x}$, the inclusion $F\subset F'$ holds. Equivalently, the minimal face of $Q$ at $\bar{x}$ is the unique face of $Q$ containing $\bar{x}$ in its relative interior.

In the current work, we will need to consider faces of epigraphs of functions. Therefore to ease notation and make the language more transparent, we extend the notion of a face to the functional setting by means of epigraphical geometry.

\begin{defn}[Faces of functions]\label{defn:face_gen}\hfill \\
{\rm 
Consider an lsc, convex function $f\colon\E\to\overline{\R}$. 
Then a set $F\subset\E$ is a {\em face} of $f$ whenever $\gph f\big|_F$ is a face of $\epi f$. A face $F$ is {\em minimal} at a point $\bar{x}\in F$ if for any other face $F'$ of $f$ containing $\bar{x}$, the inclusion $F\subset F'$ holds.}
\end{defn}
{\em Extreme points} and {\em extreme rays} of functions are simply the points and rays that happen to be faces. It is important to note that not all faces of the epigraph yield faces of the function, since such faces may contain points above the graph. 
The following simple lemma illuminates this situation.
\begin{lem}[Faces of epigraphs]\label{lem:prep} \hfill \\
Consider an lsc, convex function $f\colon\E\to\overline{\R}$. Then a face $\widehat{F}$ of the epigraph $\epi f$ contains a point $(\bar{x},r)$ with $r > f(\bar{x})$ if and only if the recession cone of $\widehat{F}$ contains the ray $\{0\}\times\R_+.$
Consequently if $\widehat{F}$ is a minimal face of $\epi f$ at  a pair $(\bar{x},f(\bar{x}))$, then $\widehat{F}$ coincides with $\gph f\big|_F$ for some set $F$ in $\E$.
\end{lem}
\begin{proof}
The implication $\Leftarrow$ is immediate.
To see the converse, let $\widehat{F}$ be a face of $\epi f$ containing a point $(\bar{x},r)$ with $r > f(\bar{x})$. Then the defining property of a face implies that $\widehat{F}$ contains the ray $\{(\bar{x},\alpha): \alpha \geq f(\bar{x})\}$. The result follows.\qed
\end{proof}

Consider an lsc, convex function $f\colon\E\to\overline{\R}$. A number of properties of faces are now immediate from the previous lemma.  To illustrate, any face of $f$ is a closed convex set and any point in the domain of $f$ is contained in some face of $f$.
Moreover, the following are equivalent for any face $F$ of $f$.
\begin{itemize}
\item $F$ is a minimal face of $f$ at $\bar{x}$,
\item $\gph f\big|_F$ is a minimal face of $\epi f$ at $(\bar{x},f(\bar{x}))$,
\item $\bar{x}$ lies in the relative interior of $F$.
\end{itemize}

The key to a facial calculus is a chain rule for a composition of a convex function and a linear mapping. We will establish this rule by bootstrapping the following result, describing faces of a preimage of  a convex set under a linear mapping \cite{handbook}.
\begin{thm}[Faces of preimages of sets]\label{thm: pre_set} \hfill \\
Consider a linear operator $\mathcal{A}\colon\E\to{\bf H}$ and a closed convex set $Q\subset{\bf H}$. Then
$F$ is a face of the preimage $\mathcal{A}^{-1}(Q)$ if and only if $F$ has the form $\mathcal{A}^{-1}(M)$ for some face $M$ of $Q$. Moreover, if $F$ is a face of $\mathcal{A}^{-1}(Q)$, then it can be written as $\mathcal{A}^{-1}(M)$, where $M$ is the minimal face of $Q$ containing $\A(F)$.
\end{thm}

The chain rule, a central result of this subsection, now easily follows. 
\begin{thm}[Faces of a composition]\label{thm:chain_gen}
Consider an lsc, convex function $f\colon{\bf H}\to\overline{\R}$ and a linear operator $\mathcal{A}\colon\E\to{\bf H}$. Then $M$ is a face of $f\circ\mathcal{A}$ if and only if $M$ has the form $\mathcal{A}^{-1}(F)$ for some face $F$ of $f$.
\end{thm}
\begin{proof}
Observe we have the representation $\epi (f\circ A)=\{(x,r):(\mathcal{A}(x),r)\in\epi f\}$, or equivalently 
$$\epi (f\circ \A)=\widehat{\A}^{-1}(\epi f),$$
for the linear mapping $\widehat{\A}(x,r):=(\mathcal{A}(x),r)$. The proof will consist of adapting Theorem~\ref{thm: pre_set} to this setting. To this end,
let $F$ be a face of $f$ and define $\widehat{F}:=\gph f\big|_F$, which is by definition a face of $\epi f$. Using Theorem~\ref{thm: pre_set}, we immediately deduce that $\widehat{\A}^{-1}(\widehat{F})$ is a face of $\epi (f\circ \A)$.
On the other hand, observe $\widehat{\A}^{-1}(\widehat{F})= \gph (f\circ \A)\big|_{\A^{-1}(F)}$. Hence $\mathcal{A}^{-1}(F)$ is a face of $f\circ\A$. Conversely, let $M$ be a face of $f\circ\A$ and define $\widehat{M}:=\gph (f\circ\A)\big|_M$, which is by definition a face of $\epi (f\circ\A)$.
Let $\widehat{F}$ be the minimal face of $\epi f$ containing $\widehat{\A}(\widehat{M})= \gph f\big|_{\A(M)}$.  By Theorem~\ref{thm: pre_set}, we have the equality $\widehat{M}=\widehat{\mathcal{A}}^{-1}(\widehat{F})$. On the other hand, since $\ri \widehat{F}$ clearly intersects $\gph f$, we deduce by Lemma~\ref{lem:prep} that we can write $\widehat{F}=\gph f\big|_F$ for some face $F$ of $f$. Consequently we obtain $\widehat{M}=\gph (f\circ{\A})\big|_{\A^{-1}(F)}$ and conclude $M=\A^{-1}(F)$, as claimed.
\qed
\end{proof}

A sum rule is immediate.
\begin{cor}[Faces of a sum]\label{cor: sum_gen}
Consider lsc, convex functions $f_1\colon\E\to\overline{\R}$ and $f_2\colon\E\to\overline{\R}$.
Then $F$ is a face of the sum $f_1+f_2$ if and only if $F$ coincides with $F_1\cap F_2$ for some faces $F_1$ of $f_1$ and $F_2$ of $f_2$.
\end{cor}
\begin{proof}
Apply Theorem~\ref{thm:chain_gen} to the linear mapping $\A (x)=(x,x)$ and to the function $g(x,y)=f_1(x)+f_2(y)$.\qed
\end{proof}

We now come back full circle and establish a tight connection between faces of functions and faces of their sublevel sets.
\begin{cor}[Faces of sublevel sets]\label{cor:face_sublevel}\hfill \\
Consider a continuous, convex function $f\colon\E\to\overline{\R}$ and let $r$ be a real number which is not a minimal value of $f$. Then equality
$$\bd [f\leq r] = [f=r] \qquad\textrm{ holds},$$ and moreover
$F$ is a proper face of the sublevel set $[f\leq r]$ if and only if $F$ coincides with $M\cap [f=r]$ for some face $M$ of $f$.
\end{cor}
\begin{proof}
Since $f$ is continuous and convex, and $r$ is not a minimal value of $f$, one can easily verify 
\begin{align*}
[f\leq r]\times \{r\}&=\epi f \cap \{(x,r): x\in\E\},\\
\bd [f\leq r] &= [f=r].
\end{align*}
Apply now Corollary~\ref{cor: sum_gen} with $f_1=\delta_{\sepi f}$ and $f_2=\delta_{\{(x,r): x\in\E\}}$. \qed
\end{proof}

It will be particularly useful for us to understand faces of the gauge function. Given a closed, convex set $Q$ containing the origin, the {\em gauge} of $Q$, denoted by $\gamma_Q\colon\E\to\R$, is defined to be $\gamma_Q(x):=\inf\,\{\lambda\geq 0: x\in \lambda Q\}$. The epigraph of $\gamma_Q$ is simply $\cl\cone (Q\times \{1\})$. See e.g. \cite[Part I]{con_ter} for more details. For the sake of simplicity, we will only consider gauges of compact sets.

\begin{cor}[Faces of a gauge]\label{cor: face_g}
Consider a compact, convex set $Q\subset\E$ containing the origin in its interior, and let $\gamma_Q\colon\E\to\R$ be the gauge of $Q$. Then $F$ is a face of $\gamma_Q$ if and only if the intersection $F\cap\bd Q$ is a face of $Q$. Moreover, if $M$ is a proper face of $Q$ then $\cone M$ is a face of $\gamma_Q$.
\end{cor}
\begin{proof}
The first claim follows from Corollary~\ref{cor:face_sublevel}, while the second is easy to verify from the definitions.\qed
\end{proof}

\subsection{Exposed faces of functions}
A special class of faces plays a particularly important role in optimization. Recall that a set $F$ is an {\em exposed face} of a  convex set $Q$ if there exists a vector $v\in\E$ satisfying $F=\argmax \{\langle v,x\rangle: x\in Q\}$, or equivalently
$F=\partial \delta^{*}_Q(v)$. In this case $v$ is the {\em exposing vector} of $F$. An exposed face $F$ is a {\em minimal exposed face at} $\bar{x}\in F$ if for any other exposed face $F'$ containing $\bar{x}$, the inclusion $F\subset F'$ holds. It is easy to see that exposed faces are themselves faces, though the converse fails in general; see for example \cite[Section 19]{con_ter}. A particularly nice situation arises when a set $Q$ is {\em facially exposed}, meaning all of its faces are exposed. For example, polyhedral sets, the positive semi-definite cone, and the nuclear norm ball are facially exposed.

We will now extend the notion of an exposed face to functions. We will see however that the calculus of exposed faces is a bit more subtle than its counterpart for general faces; namely, qualification conditions enter the picture.  To illustrate, consider the two set $Q_1:=\R\times \{0\}$ and $Q_2:=\{(x,y): x\leq 0, x^2\leq y\}\cup \R^2_{+}$. Then clearly the origin is an exposed face of $Q_1\cap Q_2=\R_+\times\{0\}$ but it cannot be written as an intersection of the faces of $Q_1$ and $Q_2$. As we will see, the reason for that is twofold: $(i)$ $Q_2$  is not facially exposed and $(ii)$ the relative interiors of the two sets do not intersect.
\begin{defn}[Exposed faces of functions]{\hfill \\}
{\rm Consider an lsc, convex function $f\colon\E\to\overline{\R}$.
A set $F\subset \E$ is an {\em exposed face} of $f$ if it has the form
$F=\partial f^{*}(v)$ for some vector $v\in\E$, or equivalently
 $$F=\argmin_{x\in \E}\, \{f(x)-\langle v,x\rangle\}.$$ In this case $v$ is an {\em exposing vector} of $F$.
An exposed face $F\subset\E$ of $f$ is {\em minimal}  at $\bar{x}\in F$ if for any other exposed face $F'$ of $f$ containing $\bar{x}$ the inclusion $F\subset F'$ holds. 
}
\end{defn}
 
Of course, specializing the definition above to the indicator function of a set, we obtain the classical notions. The following theorem is in analogy to general faces of functions (Definition~\ref{defn:face_gen}). See the appendix for details.
\begin{thm}[Epigraphical coherence of exposed faces]\label{thm:epi_coh} {\hfill \\}
Consider an lsc, convex function $f\colon\E\to\overline{\R}$ and a point $\bar{x}\in\dom f$. Then the following are true.
\begin{enumerate}
\item  A set $F$ is an exposed face of $f$ with exposing vector $v$ if and only if $\gph f\big|_F$ is an exposed face of $\epi f$ with exposing vector $(v,-1)$.
\item A set $F$ is a minimal exposed face of $f$ at $\bar{x}$ if and only if $\gph f\big|_F$ is a minimal exposed face of $\epi f$ at $(\bar{x},f(\bar{x}))$.
\end{enumerate}
\end{thm}

Recall that the minimal face of a convex set $Q$ at $\bar{x}$ is the unique face of $Q$ containing $\bar{x}$ in its relative interior. A similar characterization (in dual terms) holds for exposed faces. See the appendix for a detailed proof.
\begin{thm}[Minimal exposed faces of functions] \label{thm:min_func_rep}{\hfill \\}
Consider an lsc, convex function $f\colon\E\to\overline{\R}$ and a point $\bar{x}\in\E$. Then for any vector $v\in\ri\partial f(\bar{x})$, the set $\partial f^{*}(v)$ is a minimal exposed face of $f$ at $\bar{x}$.
\end{thm}

We now record various calculus rules of exposed faces. Again the basic result in this direction is the chain rule. In contrast to the development of facial calculus, however, the key technical tool here is the subdifferential calculus, and in particular the relationship between $\partial (f\circ \A)(x)$ and $\mathcal{A}^{*}\partial f(\A x)$, where $f$ is an lsc, convex function and $\A$ is a linear transformation \cite[Theorem 23.9]{con_ter}. 
\begin{thm}[Chain rule for conjugates] \label{thm:chain_conj}{\hfill \\}
Consider an lsc, convex function $f\colon{\bf H}\to\overline{\R}$ and a linear transformation $\mathcal{A}\colon\E\to{\bf H}$, where ${\bf H}$ and $\E$ are Euclidean spaces. Let $\bar{x}$ be a point in $\E$ and consider a vector $v\in \partial f(A\bar{x})$. Then the equation
\begin{equation}\label{eqn:chain_conj}
\partial (f\circ \A)^{*}(\A^{*}v) = \A^{-1}\partial f^{*}(v) \qquad\textrm{ holds}.
\end{equation}
\end{thm}
\begin{proof}
The inclusion $\supset$ follows directly from the chain rule $\partial (f\circ \A)(x)\subset \mathcal{A}^{*}\partial f(\A x)$. To see this, consider a point $x\in \A^{-1}\partial f^{*}(v)$. Then there exists a point $z$ in ${\bf H}$ satisfying $z=\A x$ and $v\in\partial f(z)$. We successfully conclude $\A^{*} v\in \A^{*}\partial f(\A x)\subset \partial (f\circ \A)(x)$, as claimed. 

We now prove the inclusion $\subset$ in equation \eqref{eqn:chain_conj}. To this end, consider a point $x$ satisfying $\A^{*}v\in \partial (f\circ\A)(x)$. We deduce 
$$f(\A x) \geq f(\A\bar{x})+ \langle v, \A x-\A \bar{x}\rangle \geq f(\A x) +\langle \A^{*}v, \bar{x}-x \rangle+ \langle v, \A x-\A \bar{x}\rangle =f(\A x) .$$ Thus we have equality
 throughout. Consequently $\A x$ minimizes the function $y\mapsto f(y)-\langle v,y\rangle$, and so we have $v\in\partial f(\A x)$. This completes the proof.\qed
\end{proof}

\begin{thm}[Exposed faces of a composition]\label{thm: exp_comp}
{\hfill \\}
Consider an lsc, convex function $f\colon{\bf H}\to\overline{\R}$ and a linear mapping $\mathcal{A}\colon\E\to{\bf H}$, where ${\bf H}$ and $\E$ are Euclidean spaces. Then the following are true.
\begin{enumerate}
\item If $F$ is an exposed face of $f$ with exposing vector $v$, then $\A^{-1}F$ is an exposed face of $f\circ \A$ with exposing vector $\A^{*}v$.\label{item:chain1}
\item If the range of $\A$ meets $\ri(\dom f)$, then any exposed face $M$ of $f\circ\A$ can be written as $M=\A^{-1}F$ for some exposed face $F$ of $f$. \label{item:chain2}
\end{enumerate}
\end{thm}
\begin{proof}
Claim \ref{item:chain1} follows immediately from Theorem~\ref{thm:chain_conj}. Claim \ref{item:chain2} also follows from Theorem~\ref{thm:chain_conj} since the standing assumptions of claim \ref{item:chain2} imply the exact chain rule $\partial (f\circ\A)(x)=\A^{*}\partial f(\A x)$. 
\qed
\end{proof}

\begin{cor}[Exposed faces of a sum]\label{cor: sum}
Consider two lsc, convex functions $f_1\colon\E\to\overline{\R}$ and $f_2\colon\E\to\overline{\R}$.
Then the following are true.
\begin{enumerate}
\item If $F_1$ and $F_2$ are exposed faces of $f_1$ and $f_2$ with exposing vector $v_1$ and $v_2$, respectively, then $F_1\cap F_2$ is an exposed face of $f_1+f_2$ with exposing vector $v_1+v_2$.
\item If $\ri (\dom f_1)$ meets $\ri(\dom f_2)$, then any exposed face $F$ of $f_1+f_2$ can be written as $F_1\cap F_2$ for some exposed face $F_1$ of $f_1$ and $F_2$ of $f_2$.
\end{enumerate}
\end{cor}
\begin{proof}
Apply ~\ref{thm: exp_comp} to the linear mapping $\A (x)=(x,x)$ and to the separable function $g(x,y)=f_1(x)+f_2(y)$.\qed
\end{proof}

Finally we record a relationship between exposed faces of a function and exposed faces of its sublevel sets.
\begin{cor}[Sublevel sets]\label{cor:sub_exp} \hfill \\
Consider a continuous convex function $f\colon\E\to\overline{\R}$ and a real number $r$. Then the following are true.
\begin{enumerate}
\item If $F$ is an exposed face of $f$, intersecting $[f=r]$, with exposing vector $v$, then $F\cap [f=r]$ is an exposed face of $[f\leq r]$ with exposing vector v.
\item If $r$ is not the minimum value of $f$, then every exposed face $F$ of $[f\leq r]$ has the form $F=F'\cap H$, where  $F'$ is an exposed face of $\epi f$. 
\end{enumerate}
\end{cor}
\begin{proof}
Apply Theorem~\ref{cor: sum} with $f_1=\delta_{\sepi f}$ and $f_2=\delta_{\{(x,r): x\in\E\}}$.\qed
\end{proof}

\begin{cor}[Exposed faces of a gauge]\label{cor: exp_g}\hfill \\
Consider a closed convex set $Q\subset\E$ containing the origin in its interior, and let $\gamma_Q\colon\E\to\R$ be the gauge of $Q$. Then $F$ is an exposed face of $\gamma_Q$ if and only if the intersection $F\cap\bd Q$ is a face of $Q$. Moreover, if $M$ is a proper face of $Q$ then $\cone M$ is a face of $\gamma_Q$.
\end{cor}

Now recall that the {\em polar set} and the {\em support function} of a convex set $Q$ are defined by 
$$Q^{o}:=\{v: \langle v,x\rangle \leq 1 \textrm{ for all } x\in Q\},$$
and 
$$h_Q(v)=\sup\{\langle v,x\rangle: x\in Q\},$$
respectively. Now for any compact convex sets $Q_1$ and $Q_2$, containing the origin in their interior, we successively deduce using \cite[Theorem~14.5]{conv_an} the equalities
$$\gamma_{{Q_1}}+\gamma_{{Q_2}}=h_{{Q_1}^{o}}+h_{{Q_2}^{o}}=h_{{Q_1}^{o}+{Q_2}^{o}}=\gamma_{({Q_1}^{o}+{Q_2}^o)^{o}}.$$ 
Thus the sum of gauges $\gamma_{{Q_1}}$ and $\gamma_{{Q_2}}$ is itself a gauge of $({Q_1}^{o}+{Q_2}^o)^{o}$. Combining this with Corollaries~\ref{cor: sum_gen}, \ref{cor: face_g}, \ref{cor: sum}, \ref{cor: exp_g}, we immediately deduce that if $Q_1$ and $Q_2$ are facially exposed, then so is  $(Q^o_1+Q^o_2)^o$, a rather surprising fact since the sum $Q^o_1+Q^o_2$ can easily fail to be facially exposed. We record this observation in the next theorem and will often appeal to it implicitly.
\begin{cor}[Sum of gauges]\label{cor:face_exp} 
Consider two compact, convex sets $Q_1$ and $Q_2$, containing the origin in their interior. Then the sum of the gauges of $Q_1$ and $Q_2$ is the gauge of $(Q^o_1+Q^o_2)^o$. Moreover, if the sets $Q_1$ and $Q_2$ are facially exposed, then the set $(Q^o_1+Q^o_2)^o$ is facially exposed as well.
\end{cor}

\smallskip
To illustrate, consider the following example. The symbol $\R^n$ will denote $n$-dimensional Euclidean space. The $l_p$-norm on $\R^n$ will be denoted by $\|\cdot \|_p$. 
\begin{exa}[$l_1+l_{\infty}$ norm]
{\rm
Consider the norm on $\R^n$ given by 
$$\|x\|_{1.\infty}:=\|x\|_1+\|x\|_{\infty}.$$ 
Clearly the minimal face of $\|\cdot \|_{1.\infty}$ at the origin is the origin itself. Consider now a point $\bar{x}\neq 0$. Since the $l_1$ and the $l_{\infty}$ norms are invariant under coordinate change of sign, we may suppose $\bar{x}\geq 0$. 
Define the index sets
$$\bar{I}:=\{i: \bar{x}_i=0\} \quad\textrm{ and } \quad \bar{J}:=\{i: \bar{x}_i=\|\bar{x}\|_{\infty}\}.$$ 
Then $$F:=\{x\geq 0: x_i=0 \textrm{ for each } i\in \bar{I}\}$$
is a minimal face of the $l_1$-norm at $\bar{x}$. Similarly
$$G:= \{x: x_i=\|x\|_{\infty} \textrm{ for all } i\in\bar{J} \}$$ is a minimal face of the $l_{\infty}$-norm at $\bar{x}$.
Thus $\R_{+}\{\bar{x}\}$ is an extreme ray of $\|\cdot\|_{1,\infty}$ if and only if $F\cap G$ is 1-dimensional, that is when we have $\bar{I}\cup \bar{J}=n$.  Using Corollary~\ref{cor: face_g}, one can now verify that the extreme points of the ball $\{x:\|x\|_{1,\infty} \leq 1\}$ are the points
$\prod^{k}_{i=1} \{(1+k)^{-1}\} \times \prod^n_{i=k+1} \{0\}, \textrm{ for } k=1,\ldots,n-1,$
and their images under signed permutations of coordinates; see Figure~\ref{fig:vec_norms}. 
\vspace{-10pt}
\begin{figure}[h]
        \centering
        \begin{subfigure}[b]{0.3\textwidth}
                \centering
                \includegraphics[width=\textwidth]{i}
                \caption{$l_1$-ball}
                \label{fig:gull}
        \end{subfigure}%
        ~ 
        \begin{subfigure}[b]{0.3\textwidth}
                \centering
                \includegraphics[width=\textwidth]{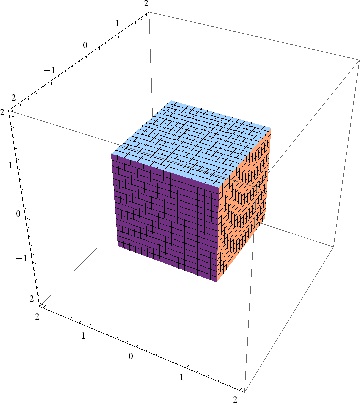}
                \caption{$l_{\infty}$-ball}
                \label{fig:tiger}
        \end{subfigure}
        ~ 
        \begin{subfigure}[b]{0.3\textwidth}
                \centering
                \includegraphics[width=\textwidth]{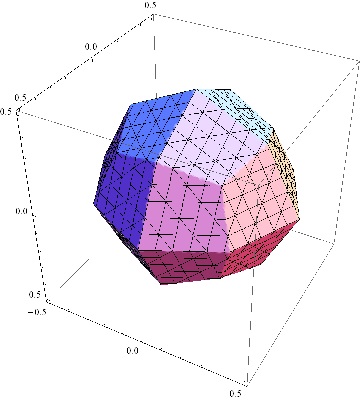}
                \caption{$(l_1+l_{\infty})$-ball}
                \label{fig:mouse}
        \end{subfigure}
       \caption{$l_1+l_{\infty}$ norm}
        \label{fig:vec_norms}
\end{figure}

}
\end{exa}


\section{Faces of the rank sparsity ball}\label{sec:rank_sp}
This section has a dual purpose: $(i)$ to shed light on the extreme points and vertices of the rank sparsity ball (see definition below) and $(ii)$ to illustrate using the rank sparsity ball how one can generally apply the facial calculus developed in the previous section to derive inequalities at the extreme points of the solution set, balancing the features of the optimization problem at hand.

We begin with some notation. The symbol  ${\bf M}^{n,m}$ will denote the space of $n\times m$-matrices, while ${\bf S}^n$ will denote the space of $n\times n$ symmetric matrices. For simplicity, in the case of ${\bf M}^{n,m}$ we will always assume $n \leq m$. We will endow ${\bf M}^{n,m}$, with the trace inner product $\langle A,B \rangle =\tr(A^{T}B)$, whose restriction is an inner product on ${\bf S}^n$.
We also define the singular value map $\sigma\colon {\bf M}^{n,m} \to \R^{n}$ taking a matrix $A$ to its vector of singular values $(\sigma_1(A),\ldots,\sigma_n(A))$ in non-increasing order. 
The group of $n\times n$ orthogonal matrices is written as ${\bf O}^n$. For any matrix $X\in {\bf M}^{n\times m}$, we consider the entry-wise $l_1$-norm $\|X\|_1:= \sum_{i,j} |X_{i,j}|$, the nuclear norm $\|X\|_*:= \sum^n_{i=1} \sigma_i(X)$, and the sum $\|X\|_{1,*}:=\|X\|_1+\theta\|X\|_{*}$ for $\theta >0$. The corresponding closed unit balls will be denoted by $\B_{1}$, $\B_{*}$, and $\B_{1,*}$, respectively. The latter is what we call the rank sparsity ball. It is clear from the previous section that the facial structure of $\|X\|_{1,*}$ does not depend on $\theta$. Consequently without loss of generality, we will set $\theta=1$ throughout. We begin the development with the following observation.
\begin{thm}[Facial exposedness]
The ball $\B_{1,*}$ is facially exposed.
\end{thm}
\begin{proof}
It is well known that $\B_{1}$ and $\B_{*}$ are facially exposed. The result now follows from Corollary~\ref{cor:face_exp}. \qed
\end{proof}
Hence there is no distinction between faces and exposed faces of $\B_{1,*}$. We will use this implicitly. The following theorem characterizes the dimension of minimal faces of the rank sparsity ball, and derives simple lower bounds on this quantity in terms of the rank and sparsity of the matrix in question. 
\begin{thm}[Faces of the ball $\B_{1,*}$] \label{thm:exposed_main}{\hfill \\}
Consider a nonzero matrix $\overline{X}\in \B_{1,*}$ along with a singular value decomposition $\overline{X}=\overline{U}(\Diag\sigma(\overline{X})) \overline{V}^T$ for orthogonal matrices $\overline{U}\in {\bf O}^{n}$ and $\overline{V}\in {\bf O}^m$.  Denote by $\bar{r}$ the rank of $\overline{X}$, and let $\widehat{U}$ and $\widehat{V}$ be the restrictions of $\overline{U}$ and $\overline{V}$ to the first $\bar{r}$ columns.
Define the set $\bar{I}:=\{(i,j): \overline{X}_{i,j}=0\}$. 
Then, the minimal face $F$ of the ball $\B_{1,*}$ at $\overline{X}$ satisfies the equation
\begin{equation}\label{eqn:char_extr}
\frac{\bar{r}(\bar{r}+1)-2}{2}- \dim\spann \{\widehat{U}^T_{i,\cdot}\widehat{V}_{j,\cdot}+\widehat{V}^T_{j,\cdot}\widehat{U}_{i,\cdot} : (i,j)\in\bar{I}\}  = \dim F.
\end{equation}
In particular, any face of the ball $\B_{1,*}$ containing $\overline{X}$ satisfies the inequality
$$\frac{\bar{r}(\bar{r}+1)}{2} -|\bar{I}|\leq \dim F+1.$$
\end{thm}
\begin{proof}
It follows from \cite[Example~5.6]{Sa_ball} and Corollary~\ref{cor: face_g} that a subset $F\subset\R^n$ is a face of the nuclear norm $\|\cdot\|_{*}$ if and only if it has the form
\[\Big\{U\left[ \begin{array}{cc}
A & 0 \\
0 & 0 
\end{array} \right] V^T: A\in {\bf S}^{k}_{+}\Big\}\] 
for $k=1,\ldots, n$ and orthogonal matrices $U\in {\bf O}^{n}$ and $V\in {\bf O}^m$. 
Clearly then $\overline{X}$ is contained in the relative interior of the face 
\[F:= \Big\{ \overline{U}\left[ \begin{array}{cc}
A & 0 \\
0 & 0 
\end{array} \right] \overline{V}^T: A\in {\bf S}^{\bar{r}}_{+}\Big\},\]
thereby making it a minimal face of $\|\cdot\|_{*}$ at $\overline{X}$. Let $\widetilde{F}$ be the minimal face of 
$\|\cdot\|_{1,*}$ at $\overline{X}$.
Then using Corollary~\ref{cor: sum_gen} and \cite[Theorem~6.5]{con_ter}, we deduce
that the affine span of $\widetilde{F}$ is the set
\[\Big\{X\in {\bf M}^{n\times m}: X_{i,j}=0 \textrm{ for all } (i,j)\in \bar{I}\Big\} \bigcap \Big\{ \overline{U}\left[ \begin{array}{cc}
A & 0 \\
0 & 0 
\end{array} \right] \overline{V}^T: A\in {\bf S}^{\bar{r}}\Big\}.\] 
Observe now that for any $A\in {\bf S}^{\bar{r}}$, we have
$$e_i^{T}\overline{U}\left[ \begin{array}{cc}
A & 0 \\
0 & 0 
\end{array} \right] \overline{V}^T e_j=\tr(\widehat{U}_{i,\cdot} A \widehat{V}^T_{j,\cdot})=\langle A, \widehat{V}^T_{j,\cdot}\widehat{U}_{i,\cdot} \rangle = \frac{1}{2}\langle A, \widehat{U}^T_{i,\cdot}\widehat{V}_{j,\cdot}+\widehat{V}^T_{j,\cdot}\widehat{U}_{i,\cdot} \rangle.$$
Applying the classical rank-nullity theorem, we deduce
$$\frac{\bar{r}(\bar{r}+1)}{2}- \dim\spann \{ \widehat{U}^T_{i,\cdot}\widehat{V}_{j,\cdot}+\widehat{V}^T_{j,\cdot}\widehat{U}_{i,\cdot}  : (i,j)\in\bar{I}\}  = \dim \widetilde{F}.$$
On the other hand, observe 
$$\dim\spann \{ \widehat{U}^T_{i,\cdot}\widehat{V}_{j,\cdot}+\widehat{V}^T_{j,\cdot}\widehat{U}_{i,\cdot}  : (i,j)\in\bar{I}\} \leq |\bar{I}| .$$
Applying Corollary~\ref{cor: face_g}, the result follows.\qed
\end{proof}

\begin{exa}
{\rm
Consider the rank two matrix $\overline{X}\subset {\bf M}^{3\times 3}$ defined by
\[\left[ \begin{array}{ccc}
1 & 1 &0 \\
0 & 1 & 1 \\
0 & 0 & 0
\end{array} \right] =  \left[ \begin{array}{ccc}
\frac{1}{\sqrt{2}} & -\frac{1}{\sqrt{2}} &0 \\
\frac{1}{\sqrt{2}} & \frac{1}{\sqrt{2}} & 0 \\
0 & 0 & 1
\end{array} \right]
\left[ \begin{array}{ccc}
\sqrt{3} & 0 &0 \\
0 & 1 & 0 \\
0 & 0 & 0
\end{array} \right]
\left[ \begin{array}{ccc}
\frac{1}{\sqrt{6}} & -\frac{1}{\sqrt{2}} & \frac{1}{\sqrt{3}} \\
\sqrt{\frac{2}{3}} & 0 & -\frac{1}{\sqrt{3}} \\
\frac{1}{\sqrt{6}} & \frac{1}{\sqrt{2}} & \frac{1}{\sqrt{3}}
\end{array} \right]^T
\] 
Then the matrices 
\[\left[ \begin{array}{cc}
\frac{1}{\sqrt{2}} & -\frac{1}{\sqrt{2}}
\end{array} \right]^T \left[ \begin{array}{cc}
\frac{1}{\sqrt{6}} & \frac{1}{\sqrt{2}}
\end{array} \right]+
\left[ \begin{array}{cc}
\frac{1}{\sqrt{6}} & \frac{1}{\sqrt{2}}
\end{array} \right]^T
\left[ \begin{array}{cc}
\frac{1}{\sqrt{2}} & -\frac{1}{\sqrt{2}}
\end{array} \right]= \left[ \begin{array}{cc}
\frac{1}{\sqrt{3}} & \frac{1}{2}-\frac{1}{\sqrt{12}}  \\
\frac{1}{2}-\frac{1}{\sqrt{12}}  & -1
\end{array} \right]
\]
and 
\[\left[ \begin{array}{cc}
\frac{1}{\sqrt{2}} & \frac{1}{\sqrt{2}}
\end{array} \right]^T \left[ \begin{array}{cc}
\frac{1}{\sqrt{6}} & -\frac{1}{\sqrt{2}}
\end{array} \right]+
\left[ \begin{array}{cc}
\frac{1}{\sqrt{6}} & -\frac{1}{\sqrt{2}}
\end{array} \right]^T
\left[ \begin{array}{cc}
\frac{1}{\sqrt{2}} & \frac{1}{\sqrt{2}}
\end{array} \right]= \left[ \begin{array}{cc}
\frac{1}{\sqrt{3}} & \frac{1}{\sqrt{12}} -\frac{1}{2} \\
\frac{1}{\sqrt{12}}-\frac{1}{2} & -1
\end{array} \right]
\] 
are linearly independent. It follows from equation (\ref{eqn:char_extr}) of Theorem~\ref{thm:exposed_main} that $\overline{X}$, up to a rescaling, is an extreme point of the ball $\B_{1,*}$.
On the other hand, a similar computation shows that the matrix 
\[\left[ \begin{array}{ccc}
1 & 1 &0 \\
0 & 1 & 0 \\
0 & 0 & 0
\end{array} \right]\]
is not an extreme point of $\B_{1,*}$ under any scaling. This is noteworthy since this matrix has the same rank as $\overline{X}$, while being more sparse, and therefore is ``preferable'' to $\overline{X}$, even though it fails to be extreme.
}
\end{exa}

As a direct consequence of the previous theorem, we now prove that when using the joint norm $\|\cdot\|_{1,*}$ to recover a point satisfying linear measurements, there is an implicit relationship at any extreme point of the solution set between the rank, sparsity, and the number of linear measurements. In contrast to the usual compressed sensing results, this relationship is absolute, being independent of noise.
\begin{thm}[Sparse low rank solutions of a linear system]\label{thm:gabor}\hfill \\
Consider the optimization problem 
\begin{align*}
\min~ &\|X\|_1 +\theta \|X\|_{*} \\
s.t.~~ &\mathcal{A}(X)=b,
\end{align*}
where $\mathcal{A}\colon {\bf M}^{n\times m}\to\R^d$ is a linear operator and $\theta$ is a strictly positive constant. Then any extreme point $\overline{X}$ of the solution set satisfies the inequality
$$\frac{\bar{r}(\bar{r}+1)}{2}-|\bar{I}| \leq 1 +d,$$
where $\bar{r}$ is the rank of $\overline{X}$ and $\bar{I}$ is the index set of the zero entries of $\overline{X}$.
\end{thm}
\begin{proof}
The constant $\theta$ will not play a role in the argument, and so we will assume $\theta=1$. Let $c$ be the optimal value of this problem. If this value is zero, then the inequality is trivially true. Hence we suppose $c >0$. The exact value of $c$ will not play a role and so for notational convenience we assume $c=1$. Define now the set $\mathcal{L}:=\{X: \mathcal{A}(X)=b\}$. Then the solution set is simply
$\B_{1,*}\cap \mathcal{L}$. Moreover, denoting the minimal face of $\B_{1,*}$ at $\overline{X}$ by $F$, we have $\{\overline{X}\}=F\cap \mathcal{L}$. We immediately deduce, 
$\dim F +\dim \mathcal{L} \leq mn$,
which using Corollary~\ref{thm:exposed_main} implies the inequality, $\frac{\bar{r}(\bar{r}+1)}{2}-|\bar{I}| \leq 1 +d,$
as claimed.
\qed
\end{proof}

\begin{rem}{\rm
We note that the extreme point inequality in Theorem~\ref{thm:gabor}, also holds with $d$ replaced by the dimension of the range of $\A$.}
\end{rem}

We next consider the vertices of the ball $\B_{1,*}$. Recall that a point $\bar{x}$ of a convex set $Q$ is a {\em vertex} if the normal cone $N_Q(\bar{x})$ is full-dimensional. In particular, the set of exposing vectors of such points has nonzero Lebesgue measure. It is standard that the vertices of $\B_1$ are simply its extreme points, while the ball $\B_{*}$ has no vertices. We will shortly see that remarkably the rank sparsity ball $B_{1,*}$ has no ``new'' vertices, that is all of its vertices  are simply the extreme points of $\frac{1}{2}\B_{1}$. The following lemma is key, and may be of independent interest. It establishes certain lower-bounds on the size of the set of all matrices with a prescribed rank and sparsity pattern.

\begin{lem}[Sparsity-rank intersections]\label{lem:sparse_rank}\hfill \\
Consider a matrix $\overline{X}\in{\bf M}^{n\times m}$ and let $r=\rank \overline{X}$ and $\bar{I}:=\{(i,j):\overline{X}_{i,j}=0\}$. Define a bipartite graph $\mathcal{G}$ on $n\times m$ vertices with the edge set $\bar{I}^c$, and denote by $c(\mathcal{G})$ the number of connected components of $\mathcal{G}$. Then there exists a ${\bf C}^{\infty}$ manifold $\mathcal{M}$ of dimension $n+m-c(\mathcal{G})$ satisfying
\begin{equation}\label{eqn:inclusion_man}
\overline{X}\in\mathcal{M}\subset \{X: \rank X = r ~\textrm{ and }~ X_{i,j}=0 \textrm{ for all } (i,j)\in \bar{I}\}.
\end{equation}
Moreover, letting $\alpha$ be the number of nonzero rows of $\overline{X}$ and $\beta$ be the number of nonzero columns of $\overline{X}$, there exists a linear subspace $ \mathcal{V}$ of  ${\bf M}^{n\times m}$ satisfying  
\begin{equation}\label{eqn:inclusion}
\overline{X}\in\mathcal{V}\subset \cl \mathcal{M} \subset \{X: \rank X \leq r ~\textrm{ and }~ X_{i,j}=0 \textrm{ for all } (i,j)\in \bar{I}\},
\end{equation}
and having $\dim\mathcal{V}\geq \max\{\alpha,\beta\}$.
\end{lem}
\begin{proof}
Consider the Lie group action of $GL(n)\times GL(m)$ on ${\bf M}^{n\times m}$ defined by 
$\theta_{(U,V)}(X) := U X V^T$. 
Restricting this action to diagonal matrices, we obtain an action of $(\R\setminus \{0\})^{n}\times (\R\setminus \{0\})^{m}$ on ${\bf M}^{n\times m}$ defined by $\widehat{\theta}_{(u,v)}(X) := \Diag(u) X \Diag(v)$. Let $\mathcal{M}$ be the orbit of $\overline{X}$ under $\widehat{\theta}$, namely set 
$$\mathcal{M}:=\{\Diag(u) \overline{X} \Diag(v): u\in (\R\setminus \{0\})^{n}\textrm{ and } v\in(\R\setminus \{0\})^{m} \}.$$
Clearly inclusions \eqref{eqn:inclusion_man} hold.
By \cite[Proposition~7.26]{Lee2}, we deduce that the mapping $F(u,v):= \Diag(u) \overline{X} \Diag(v)$ has constant rank.
Moreover, since the orbits of semi-algebraic Lie group actions are always ${\bf C}^{\infty}$-smooth manifolds (see \cite[Theorem B4]{Gib}), we deduce that $\mathcal{M}$ is a ${\bf C}^{\infty}$-smooth manifold with dimension equal to the rank of the linear operator $DF(e,e)\colon \R^n\times\R^m \to {\bf M}^{n\times m}$. Observe, on the other hand, that we have
$$DF(e,e)(v,w)= \Diag(v)\overline{X}+\overline{X}\Diag(w).$$
Hence equality $DF(e,e)(v,w)=0$ holds if and only if  we have $u_i=-v_j$ for all $(i,j)\notin \bar{I}$.  
It follows immediately that the kernel of the operator $DF(e,e)(v,w)$ has dimension $c(\mathcal{G})$, and therefore that $\mathcal{M}$ is $n+m-c(\mathcal{G})$ dimensional, as claimed.

Now let $\mathcal{R}$ consist of all indices $i$ such that the $i$'th row of $\overline{X}$ is nonzero. 
Choose an arbitrary index $i^{*}\in\mathcal{R}$ and define a vector $v_{i^{*}}=e\in\R^n$. For each index $i\in \mathcal{R}\setminus \{i^{*}\}$, choose a vector $v\in\R^n$ so that the vectors $\{v_i[\mathcal{R}]\}_{i\in\mathcal{R}}$ are linearly independent. Now for each index $i\in \mathcal{R}$,  define a matrix $B_i:=\Diag(v_i) \overline{X}$ and let $\mathcal{V}:=\spann \{B_i: i\in\mathcal{R}\}$. Clearly inclusions (\ref{eqn:inclusion}) hold.
We claim that the matrices $B_i$ are all linearly independent. Indeed suppose there are numbers $\lambda_i$ for $i\in\mathcal{R}$ satisfying $$0=\sum_{i\in\mathcal{R}} \lambda_i B_i= \Diag \Big(\sum_{i\in\mathcal{R}} \lambda_i v_i\Big)\overline{X}.$$
Hence $\sum_{i\in\mathcal{R}} \lambda_i v_i[\mathcal{R}]=0$, and we conclude $\lambda_i=0$ for each index $i\in\mathcal{R}$. Applying an 
analogous argument to $\overline{X}^T$, the result follows.\qed
\end{proof}

\begin{thm}[Vertices of the ball $\B_{1,*}$] \label{thm:vert}{\hfill \\}
The vertices of the ball $\B_{1,*}$ are simply the extreme points of $\frac{1}{2}\B_{1}$, that is 
matrices having all zero entries except for one entry whose value is $\pm\frac{1}{2}$.
\end{thm}
\begin{proof}
First observe that  a matrix $\overline{X}$ is a vertex of the ball $\B_{1,*}$ if and only if the equation $\|\overline{X}\|_{1,*}= 1$ holds and the set $\partial\|\cdot\|_{1,*}(\overline{X})$ has dimension $n-1$. Consequently any matrix having all zero entries except for one entry whose value is $\pm\frac{1}{2}$ is a vertex of $\B_{1,*}$. We will now show that these are the only vertices of this ball. To this end, suppose that $\overline{X}$ is a vertex of $\B_{1,*}$, and define $r=\rank \overline{X}$ and $\bar{I}:=\{(i,j):\overline{X}_{i,j}=0\}$.

We claim that the equation
\begin{equation}\label{eqn:inter_main}
\{X: \rank X = r ~\textrm{ and }~ X_{i,j}=0 \textrm{ for all } (i,j)\in \bar{I}\} = \R_{++}\{\overline{X}\},
\end{equation}
holds locally around $\overline{X}$. To see this, suppose not. Then there exists a sequence $X^k$ with $X^k\notin \R_{++}\{\overline{X}\}$ for all $k$, and satisfying 
$X^{k}\to\overline{X}$, $\rank X^k = r$, and $X^k_{i,j}=0$ for all $(i,j)\in\overline{I}$.
Choose a vector $\overline{V}\in\ri \partial \|\cdot\|_{1,*}(\overline{X})$. It is standard that the set-valued mapping $X\mapsto \partial \|\cdot\|_1(X)$ is inner-semicontinuous at $\overline{X}$ relative to the linear space $\{X: X_{i,j}=0 \textrm{ for all } (i,j)\in \bar{I}\}$. Similarly $X\mapsto \partial \|\cdot\|_{*}(X)$ is inner-semicontinuous at $\overline{X}$ relative to the manifold $\{X: \rank X = r\}$. It follows that  there exists a sequence $V^k\in \partial \|\cdot \|_{1,*}(X^k)$ converging to $\overline{V}$. Hence, the points $\frac{1}{\|X^k\|_{1,*}}X^k$ converge to $\overline{X}$ and the vectors $V^k\in N_{\B_{1,*}}\big(\frac{1}{\|X^k\|_{1,*}}X^k\big)$ converge to the vector $\overline{V}$ lying in the interior of $N_{\B_{1,*}}(\overline{X})$, which is a contradiction.
Thus equation (\ref{eqn:inter_main}) holds.
On the other hand, Lemma~\ref{lem:sparse_rank} along with lower-semicontinuity of the rank function implies that $\overline{X}$ must have at most one nonzero row and at most one nonzero column, as claimed. \qed
\end{proof}

\section{Recovering sparse rank one matrices with the joint norm}\label{sec:rec}
In this section, we will prove that the problem of minimizing $\langle V, X\rangle$ subject to $\Vert X\Vert_1 + \theta \Vert X\Vert_*\le 1$ will
recover a sparse rank-one matrix for a positive measure subset of matrices $V$; see Theorem~\ref{thm:sp_rank}. Indeed, this property is key for the results of Doan and Vavasis \cite{DV13} and Doan, Toh and Vavasis \cite{prox_norm}, who use the joint norm $\| \cdot\|_{1,*}$ to find hidden rank-one blocks inside large matrices. We will elaborate on the significance of this result further at the end of this section. We begin with the following key lemma, which may be of an independent interest. Roughly speaking, it shows that any translate of any open subregion of the smooth part of the boundary of the spectral norm ball generates, by way of the positive hull operation, a region with nonempty interior. The proof requires some elementary differential geometry; see for example~\cite{Lee2}. In particular, we say that a smooth mapping between smooth manifolds is a {\em submersion} at a point if the derivative of the mapping there is surjective.

\begin{lem}[Positive hull of the translated spectral ball]\label{lem: con_spec} \hfill\hfill \\ 
Consider the analytic manifold
$$\mathcal{M}:=\{X \in {\bf M}^{n\times m}: 1=\sigma_1(X)> \sigma_2(X) \geq \ldots\geq \sigma_{n}(X)\},$$
and fix a matrix $Y\in  {\bf M}^{n\times m}$. Then the positive scaling mapping 
\begin{align*}
\Phi\colon \R_{++}\times \mathcal{M} &\to {\bf M}^{n\times m}, \\
(\alpha,X) &\mapsto \alpha (Y+X),
\end{align*}
is a submersion at a pair $(\alpha, X)$ if and only if the condition
$$\langle Yv, u\rangle \neq -1 \qquad\textrm{ holds},$$
where $u$ and $v$ are the left and  right singular vectors of $X$ corresponding to $\sigma_1(X)$, appearing in any singular value decomposition of $X$.
Consequently there exists a dense subset $\mathcal{D}_{Y}$ of $\mathcal{M}$ so that $\Phi$ is a submersion at any point in $\R_{++}\times\mathcal{D}_{Y}$. Therefore  
for any open set $W$ that intersects $\mathcal{M}$, the set $\R_+\big(Y+(\mathcal{M}\cap W)\big)$ has nonempty interior.
\end{lem}
\begin{proof}
Define the mapping $\Phi$ as in the statement of the theorem.
A trivial computation shows that for any $\alpha >0$ and $X\in \mathcal{M}$, we have 
$$\rge D\Phi(\alpha,X)=\spann(Y+X) +T_{\mathcal{M}}(X),$$
where $T_{\mathcal{M}}(X)$ denotes the tangent space to $\mathcal{M}$ at $X$.
It is standard that $\mathcal{M}$ has codimension $1$ and the normal space has the form
$$N_{\mathcal{M}}(X)= \spann (uv^{T}),$$ 
where $u$ and $v$ are the left and right singular vectors of $X$ corresponding to $\sigma_1(X)$, appearing in the singular value decomposition $X=U(\Diag\sigma(X))V^T$. This formula immediately follows for example from~\cite[Theorem~7.1]{LHSing}.
Now $\Phi$ is a submersion at $(\alpha, X)$ if and only if we have $Y+X\notin T_{\mathcal{M}}(X)$, or equivalently
$$\langle Y+X, uv^{T}\rangle \neq 0.$$
Expanding, we obtain 
\begin{align*}
0\neq \langle Y+&U(\Diag\sigma(X))V^T, uv^{T}\rangle= \\
&=\langle Yv,u\rangle +\tr \big((\Diag\sigma(X)) U^Tu v^TV\big) = \langle Yv, u\rangle +1. 
\end{align*}
This proves the first assertion of the theorem. Consequently if $\Phi$ is not a submersion at $(\alpha, X)$, keeping $v$ fixed, we may rotate $u$ slightly to a new vector $\hat{u}$ so that $\langle Yv, \hat{u}\rangle\neq -1$. Denote the corresponding rotation matrix by $R$. Then the matrix $\widehat{X}:= RX$ lies in $\mathcal{M}$ close to $X$, while $\Phi$ is indeed a submersion at $(\alpha, \widehat{X})$. 
Thus there exists a dense subset $\mathcal{D}_{Y}$ of $\mathcal{M}$ so that $\Phi$ is a submersion at any point in $\R_{++}\times\mathcal{D}_{Y}$. Consequently applying the open mapping theorem, we conclude that $\Phi$ sends open sets to open sets. The result follows.
\qed
\end{proof}

We now arrive at the main result of this section. In what follows, for any matrix $X\in {\bf M}^{n\times m}$ we define the index set $\supp X:=\{(i,j): X_{i,j}\neq 0\}$.
\begin{thm}[Sparsity and rank one matrices]\label{thm:sp_rank}\hfill\\
Consider a rank one matrix $\overline{X}\in {\bf M}^{n\times m}$. Then after a permutation of rows and columns it has the form 
\begin{equation}\label{eqn: rank1}
\left[ \begin{array}{ccc}
A & {\bf 0}_{p,m-q}  \\
{\bf 0}_{n-p,q} & {\bf 0}_{n-p,m-q} \end{array} \right]
\end{equation}
for some rank one matrix $A\in {\bf M}^{p\times q}$ with all nonzero entries. Consequently, the set $$\mathcal{K}:=\{X: \rank X = 1 ~\textrm{ and }~ \supp X=\supp \overline{X})\}$$
is a $(p +q -1)$-dimensional analytic manifold. Furthermore, there is a set of matrices $V$ of positive measure such that 
the problem 
\begin{equation}\label{eqn: opt_prob}
\min\, \{\langle V,X\rangle : \|X\|_1+\|X\|_*\leq 1\}
\end{equation}
admits a unique minimizer and this minimizer lies in 
$\mathcal{K}$.
\end{thm}
\begin{proof}
Observe that $\overline{X}$ can be factored as $\overline{X}=uv^{T}$ for some vectors $u\in\R^n$ and $v\in\R^m$. Consequently if an entry $\overline{X}_{i,j}$ is zero, then either the whole $i$'th row or the whole $j$'th column of $\overline{X}$ is zero. Hence we may permute the rows and columns of $\overline{X}$ so that the resulting matrix has the form \eqref{eqn: rank1}.

We will assume without loss of generality $p \leq q$. It is standard that for almost every matrix $V$, the problem \eqref{eqn: opt_prob} has a unique solution. Consequently it is sufficient to show that the set
$$\bigcup_{X\in \mathcal{K}} N_{\B_{1,*}}(X)= \bigcup_{X\in \mathcal{K}} \R_+\partial \|\cdot\|_{1,*}(X),$$
has nonzero measure. (Equality above follows from say \cite[Corollary~23.7.21]{con_ter}.)
Before we proceed with the rest of the proof, we recall (see for example \cite[Theorem~7.1]{LHSing}) that the subdifferential of the nuclear norm at any matrix $X$ is given by 
\begin{align*}
\partial \|\cdot\|_{*}(X)=\{U(\Diag &w) V^T:~ w\in\partial \|\cdot\|_{1}(\sigma(X)) \textrm{ and }\\
 &X=U(\Diag (\sigma(X))V^T \textrm{ with } U\in {\bf O}^n,  V\in {\bf O}^m \}.
\end{align*}
Consider now a matrix $X$ of the form (\ref{eqn: rank1}) and let $\|\cdot\|^{p,q}_{1}$ and $\|\cdot\|^{p,q}_{*}$ be the restrictions of $\|\cdot\|_1$ and $\|\cdot\|_{*}$ to ${\bf M}^{p,q}$, respectively. We claim that the inclusion 
$$\Big\{\left[ \begin{array}{ccc}
Y & {\bf 0}  \\
{\bf 0} & {\bf 0} \end{array} \right]
: Y\in \partial \|\cdot\|^{p,q}_{*}(A)\Big\}\subset \partial \|\cdot\|_{*}(X),$$
holds. To see this, consider  a matrix $Y\in \partial \|\cdot\|^{p,q}_{*}(A)$. Then there exist matrices $U\in {\bf O}^p$ and $V\in {\bf O}^q$ and a vector $y\in\partial \|\cdot\|^{p.q}_1(\sigma(A))$ satisfying  
$$Y=U(\Diag y) V^T \qquad \textrm{ and } \qquad A=U(\Diag\sigma(A)) V^T.$$
Clearly $\sigma(X)=\sigma(A)\times\{0\}^{n-p}$ and $\{y\}\times\{0\}^{n-p}\in\partial \|\cdot\|_1(\sigma(X))$.
Consequently we deduce 
$$X=\left[ \begin{array}{ccc}
U & {\bf 0}  \\
{\bf 0} & {\bf I} \end{array} \right] 
\left[ \begin{array}{ccc}
 \Diag \sigma(A) & {\bf 0}  \\
{\bf 0}  & {\bf 0} \end{array} \right]\left[ \begin{array}{ccc}
V & {\bf 0}  \\
{\bf 0} & {\bf I} \end{array} \right]^T $$
and therefore $$\left[ \begin{array}{ccc}
Y & {\bf 0}  \\
{\bf 0} & {\bf 0} \end{array} \right]
=\left[ \begin{array}{ccc}
U & {\bf 0}  \\
{\bf 0} & {\bf I} \end{array} \right] 
\left[ \begin{array}{ccc}
 \Diag y & {\bf 0}  \\
{\bf 0}  & {\bf 0} \end{array} \right]\left[ \begin{array}{ccc}
V & {\bf 0}  \\
{\bf 0} & {\bf I} \end{array} \right]^T\in \partial \|\cdot\|_{*}(X),$$
as claimed.

Now fix for the duration of the proof a matrix $X$ of the form (\ref{eqn: rank1}) and the corresponding submatrix $A$. Since all the entries of $A$ are nonzero, there is a neighborhood of $A$ on which $\|\cdot\|^{p,q}_{1}$ is smooth and moreover the gradient $\nabla \|\cdot\|^{p,q}_{1}$ is constant. Denote this neighborhood by $\mathcal{U}$ and define $H:=\nabla \|\cdot\|^{p,q}_{1}(A)$. 
Consider now any matrix 
$D:=\left[ \begin{array}{cc}
C & {\bf 0}  \\
{\bf 0} & {\bf 0} \end{array} \right]$
of the form (\ref{eqn: rank1}) with $C\in\mathcal{U}$.
From the subdifferential sum rule, we now deduce that any matrix of the form.
\begin{align}\label{eqn:mat}
\{Z\in {\bf M}^{n\times m}: \quad &Z\big[[1,p],[1,q]\big]\in H +\partial \|\cdot\|^{p,q}_{*}(C), ~\textrm{and} \\
&|Z_{ij}|\leq 1 \textrm{ for all } (i,j)\notin [1,p]\times [1,q]\}\notag
\end{align}
is contained in $\partial \|\cdot\|_{1,*}(D)$. For ease of notation, we will denote this set of matrices (\ref{eqn:mat}) by 
$$\left[ \begin{array}{ccc}
H +\partial \|\cdot\|^{p,q}_{*}(C) & {\bf \odot}  \\
{\bf \odot} & {\bf \odot} \end{array} \right].$$
Thus it is sufficient to argue that the set 
$$\Gamma:=\bigcup_{C\in \mathcal{U}: ~\scriptsize{\rank} C=1 } \R_+ \left[ \begin{array}{cc}
H +\partial \|\cdot\|^{p,q}_{*}(C) & {\bf \odot}  \\
{\bf \odot} & {\bf \odot} \end{array} \right] $$
has nonempty interior. On the other hand, the equation
\begin{equation*}
\Gamma =\R_+ \Big(~ \left[ \begin{array}{cc}
H & {\bf 0}  \\
{\bf 0} & {\bf 0} \end{array} \right]  + \bigcup_{C\in \mathcal{U}: ~\scriptsize{\rank} C=1 }  \left[ \begin{array}{cc}
\partial \|\cdot\|^{p,q}_{*}(C) & {\bf \odot}  \\
{\bf \odot} & {\bf \odot} \end{array} \right] ~\Big).
 \end{equation*}
 holds. Denote now the spectral ball in ${\bf M}^{p\times q}$ by  $\B^{p,q}_{s}:=[\sigma_1\leq 1]$. Loosely speaking, we now claim that $\bigcup \{ \ri\partial \|\cdot\|_{*}^{p,q}(C): \rank C=1\}$ 
 coincides with the smooth part of the boundary of the spectral ball $\B^{p,q}_{s}$. To see this, we appeal to \cite[Theorem~4.6]{spec_id} and obtain 
\begin{align*}
\mathcal{M}:&=\bigcup_{C:\, \scriptsize{\rank} C=1} \ri \partial \|\cdot\|_{*}^{p,q}(C)\\
&= \sigma^{-1}\Big(\bigcup_{x:\, \scriptsize{\rank} x=1} \ri \partial \|\cdot\|_1(x)\Big)\\
&= \sigma^{-1}\Big(\{x: \|x\|_{\infty}\leq 1 \textrm{ and there exists unique } i \textrm{ with } |x_i|=\|x\|_1\}\Big)\\
&= \{Q \in {\bf M}^{p\times q}: \sigma_1(Q)=1, \sigma_2(Q) <1,\ldots, \sigma_{p}(Q)<1\}.
\end{align*}

Now since the set-valued mapping $(\partial \|\cdot\|_{*}^{p,q})^{-1}=N_{\B^{p,q}_{s}}$ is inner semi-continuous when restricted to $\mathcal{M}$ (see for example \cite[Proposition~3.15]{spec_id}), a routine argument shows that  
$$\bigcup_{C\in \mathcal{U}:\, \scriptsize{\rank} C=1} \ri \partial \|\cdot\|_{*}^{p,q}(C),$$
is an open subset of $\mathcal{M}$. Hence we may write it as $\mathcal{M}\cap W$ for some open subset $W$ of $ {\bf M}^{p\times q}$.
Finally to conclude the proof, it is sufficient to show that 
the set 
\begin{equation*}
\R_+ \Big(~ \left[ \begin{array}{cc}
H & {\bf 0}  \\
{\bf 0} & {\bf 0} \end{array} \right]  +  \left[ \begin{array}{cc}
\mathcal{M}\cap W & {\bf \odot}  \\
{\bf \odot} & {\bf \odot} \end{array} \right] ~\Big)
 \end{equation*}
 has nonempty interior, but this is immediate from Lemma~\ref{lem: con_spec}.
\qed
\end{proof}

\subsection*{Discussion on sparse rank one recovery:}\label{sec:disc}
In Theorem~\ref{thm:sp_rank}, we argued that the problem of minimizing $\langle V, X\rangle$ subject to $\Vert X\Vert_1 + \theta \Vert X\Vert_*\le 1$ will
recover a sparse rank-one matrix for a positive measure subset of matrices $V$.  This result, on the other hand, is not possible with either the 1-norm or nuclear-norm solely.  In other words, minimizing
$\langle V, X \rangle$ subject to $\Vert X\Vert_1 \le 1$ will recover a sparse matrix
for a positive measure subset of matrices $V$, but will recover a
sparse rank-one matrix only for a set of matrices $V$ of measure zero (except in the somewhat trivial case when the solution has a single nonzero row or column).  The same holds for the nuclear norm alone.
Indeed, this property is key for the results of Doan and Vavasis \cite{DV13} and Doan, Toh and Vavasis \cite{prox_norm},  who used the joint norm $\| \cdot\|_{1,*}$ to find hidden rank-one blocks inside large matrices. 

Thus, the sum of the 1-norm and the nuclear norm appears to have greater power to recover sparse rank-one matrices than either norm alone.  This should be contrasted with the results of Oymak et al. \cite{Fazel_neg} who show that for the exact recovery problem given linear measurements, a sum of norms performs no better than the two norms individually.  More precisely, the authors of \cite{Fazel_neg} consider the following problem: given a linear operator $\mathcal{A}\colon {\bf M}^{n,m}\to\R^d$ and a measurement vector $b\in \R^d$, find a sparse low-rank matrix $X$ satisfying $\mathcal{A}(X)=b$.  They argue that the number of measurements (i.e. the value $d$) to guarantee recovery by minimizing $\Vert X\Vert_1 + \theta \Vert X\Vert_*$ subject to $\mathcal{A}(X)=b$ is no better than the number needed when using only one of the norms in the objective function.

It is not clear why our results point in the opposite direction of \cite{Fazel_neg}; possibly the disparity is because \cite{Fazel_neg} focuses on minimizing measurements in the noise-free case, whereas \cite{DV13,prox_norm} assume the entire matrix is known (i.e., the number of measurements is unlimited) but the data is corrupted by noise.  Indeed, at the end of \cite{Fazel_neg}, the authors note that extending their results to noisy sparse Principle Component Analysis would be an interesting direction to pursue.

\begin{acknowledgement} We thank Gabor Pataki for insightful discussions, and in particular for suggesting including Subsection~\ref{subsec:face_gen} and Theorem~\ref{thm:gabor}.

\end{acknowledgement}

\bibliographystyle{plain}

\bibliography{bibliography}

\begin{appendices}
\section{Appendix}
Recall that the minimal face of a convex set $Q$ at $\bar{x}$ is the unique face of $Q$ containing $\bar{x}$ in its relative interior.
A similar characterization holds for minimal exposed faces: any set of the form $\partial\delta^*_Q(v)$ for some vector $v\in\ri N_Q(\bar{x})$ is the minimal exposed face of $Q$ at $\bar{x}$. To be self-contained, we provide an elementary proof. We begin with the following lemma.
\begin{lem}[Exposed faces of the dual cone]\label{lem:key} 
Consider a closed convex cone $K\subset\E$ and a point $\bar{x}$ in $K$. Then the set $F:=N_K(\bar{x})$ is an exposed face of $K^{o}$ and the equality
$$N_F(\bar{v})\cap K = N_{K^{o}}(\bar{v})\quad\textrm{ holds for any } \bar{v}\in F.$$ 
\end{lem}
\begin{proof}
It follows immediately from the equality $N_K(\bar{x})=(N_{K^{o}})^{-1}(\bar{x})$ that $F$ is an exposed face of $K^{o}$. Clearly the inclusion $N_{K^{o}}(\bar{v})\subset N_F(\bar{v})\cap K$ holds. Consider now an arbitrary vector $w\in N_F(\bar{v})\cap K$. Then by Moreau's decomposition theorem (see for example \cite{Moreau}), we have the representation
$$w= w_1+w_2, \textrm{ for some } w_1\in N_{K^{o}}(\bar{v}), w_2\in T_{K^{o}}(\bar{v}) \textrm{ with } \langle w_1,w_2\rangle=0.$$
Consequently there exist vectors $v_i\in K^{o}$ and real numbers $\lambda_i >0$ with $\lambda_i(v_i-\bar{v})\to w_2$. Hence given any $\epsilon >0$, for all sufficiently large indices $i$ we have
$$0 \geq \langle w,  \lambda_i\bar{v}+ \lambda_i (v_i-\bar{v})\rangle\geq \lambda_i \langle w,\bar{v}\rangle + \langle w,w_2\rangle -\epsilon= \lambda_i \langle w,\bar{v}\rangle+\|w_2\|^2+\langle w_1,w_2\rangle-\epsilon.$$
Now observe since $F$ is a cone, we have $\langle w,\bar{v}\rangle =0$. Consequently letting $\epsilon$ tend to zero we obtain $w_2=0$, thereby completing the proof. \qed
\end{proof}

\begin{thm}[Minimal exposed faces of convex cones]\label{thm:exp} {\hfill \\}
Consider a closed, convex cone $K\subset\E$ and a point $\bar{x}$ in $K$. Then for any vector $v\in \ri N_K(\bar{x})$, the set $F=\partial \delta^{*}_K(v)$
is a minimal exposed face of $K$ at $\bar{x}$.
\end{thm}
\begin{proof}
Consider vectors $v\in\ri N_{K}(\bar{x})$ and $w\in N_{K}(\bar{x})$. Using Lemma~\ref{lem:key}, we obtain
$$(N_K)^{-1}(w)=N_{K^{o}}(w)=N_{N_{K}(\bar{x})}(w)\cap K\supset N_{N_{K}(\bar{x})}(v)\cap K=(N_K)^{-1}(v),$$
and the result follows.\qed
\end{proof}

The theorem above can easily be extended to convex sets by homogenizing; see Corollary~\ref{cor:rep_conv_set}. 
It will be particularly useful for us to understand the exposed faces of the gauge function. The proof of the following proposition is standard; we provide details for the sake of completeness.

\begin{prop}[Exposed faces of the gauge]\label{prop:gauge} {\hfill \\}
Consider a closed, convex set $Q\subset\E$ containing the origin in its interior, and let $\gamma_Q\colon\E\to\R$ be the gauge of $Q$. Then the following are true.
\begin{enumerate}
\item If $F$ is an exposed face of $Q$ with exposing vector $v$, then $\cl \cone F$ is an exposed face of $\gamma_Q$ with exposing vector $\frac{v}{\langle v,x\rangle}$, where $x$ is any point of $F$. 
\item If $F$ is an exposed face of $\gamma_Q$ with exposing vector $v\neq 0$, then $F\cap (\bd Q)$ is an exposed face of $Q$ with exposing vector $v$. Moreover $F$ then has the representation $F=\cl\cone (F\cap (\bd Q))$.
\end{enumerate}
Similarly the following are true.
\begin{enumerate}
\item[3] If $F$ is a minimal exposed face of $Q$ at $\bar{x}$, then $\cl \cone F$ is a minimal exposed face of $\gamma_Q$ at $\bar{x}$. 
\item[4] If $F$ is a minimal exposed face of $\gamma_Q$ at $\bar{x}$, so that the intersection $F\cap (Q^{\infty})^c$ is nonempty, then $F\cap (\bd Q)$ is a minimal exposed face of $Q$ at $\bar{x}$.
\end{enumerate}
Moreover, for any point $x\in \bd Q$ and nonzero vector $v\in\E$, the equivalence
\begin{equation}\label{eqn:equiv}
v\in N_Q(x) \quad \Longleftrightarrow \quad \frac{v}{\langle v,x \rangle}\in\partial \gamma_Q(x)\qquad \textrm{holds}
\end{equation}
\end{prop}
\begin{proof}
By \cite[Corollary~9.7.1]{con_ter}, we have
$$Q=\{x: \gamma_Q(x)\leq 1\}, \quad\bd Q=\{x: \gamma_Q(x)= 1\}, \quad Q^{\infty}=\{x: \gamma_Q(x)=0\},$$
where $Q^{\infty}$ is the recession cone of $Q$.
We first prove $1$. To this end, suppose that $F$ is an exposed face of $Q$ with an exposing vector $v$. Let $\bar{x}$ be an arbitrary point of $F$ and define $\beta:= \langle v,\bar{x}\rangle$. Then the inequality
$\langle v,x\rangle\leq \beta$ holds for all $x\in Q$ and we have $F=\{x\in Q: \langle v,x\rangle = \beta\}$. Since $Q$ contains $0$ in its interior, we deduce $\beta \neq 0$.
Define now the hyperplane $$H:=\{(x,\alpha)\in\E\times\R : \langle (v,-\beta), (x,\alpha) \rangle = 0\}.$$ We claim that it supports $\epi \gamma_Q$.
To see this, simply observe that for any vector $x\in Q$, we clearly have   
$\langle (v,-\beta), (x,1) \rangle\leq 0$. We deduce that $Q\times\{1\}$ lies on one side of $H$ and consequently so does $\epi\gamma_Q=\cl\cone ((0,0),Q\times\{1\})$.

Now consider a point $(x,\alpha)\in H\cap \epi \gamma_Q$. Since $\gamma_Q$ is continuous, we deduce $\alpha=\gamma_Q(x)$. Suppose first $\alpha\neq 0$. Then we have equality $\langle v,\alpha^{-1}x \rangle = \beta$. Consequently $x$ lies in  $\cone F$. Suppose on the other hand $\alpha=0$, that is $x\in Q^{\infty}$. Then it is easy to see that equality $\langle v,x \rangle = 0$ holds. Choose an arbitrary point $y\in F$. Observe $y+\lambda x$ lies in $F$ for all $\lambda \geq 0$. Hence $\frac{1}{\lambda}(y+ \lambda x)$ lies in $\cone F$ and converges to $x$ as we let $\lambda$  tend to $\infty$. We deduce $x \in \cl \cone F$.
Conversely, suppose $x$ lies in $\cone F$. Observe that $\R_{+}\{x\}$ intersects $\bd Q$ in a unique point. It then easily follows $\gamma_Q(x)\neq 0$ and consequently that $\frac{x}{\gamma_Q(x)}$ lies in $F$. We deduce $\langle v, \frac{x}{\gamma_Q(x)}\rangle=\beta$  and therefore $\langle (v,-\beta), (x,\gamma_Q(x)) \rangle = 0$. Hence $\cone (F\times \{1\})$ is contained in $H\cap \epi \gamma_Q$. Taking closure, we obtain $\cl \cone (F\times \{1\})\subset H\cap \epi \gamma_Q$. We conclude that $\cl\cone F$ is an exposed face of $\gamma_Q$ with exposing vector $\beta^{-1}v$, as claimed.

We now prove $2$. To this end, suppose that $F$ is an exposed face of $\gamma_Q$ with an exposing vector $v\neq 0$. 
Then $L:=\gph \gamma_Q\big|_F$ is an exposed face of $\epi \gamma_Q$
with exposing vector $(v,-1)$. 
Consequently the inequality
$\langle (v, -1), (x,\alpha)\rangle\leq 0$  holds for all $(x,\alpha)\in \epi \gamma_Q$ and we have $L= \{(x,\alpha)\in\epi \gamma_Q :\langle (v, -1), (x,\alpha)\rangle=0\}$. 
Define the hyperplane $H:=\{x: \langle v,x\rangle = 1\}$. It easily follows that $H$ is a supporting hyperplane of $Q$ and we have $H\cap Q = \{x: (x,1)\in L\}= F\cap \bd Q$. Thus $F\cap \bd Q$ is an exposed face of $\gamma_Q$ with an exposing vector $v$. Applying claim $1$ now to $F\cap \bd Q$, we deduce
$M=\cl\cone (M\cap (\bd Q))$.

To see 3, suppose that $F$ is a minimal exposed face of $Q$ at $\bar{x}$. Then by claim $1$, the set $\cl \cone F$ is an exposed face of $\gamma_Q$ containing $\bar{x}$. Consider now any exposed face $M$ of $\gamma_Q$ containing $\bar{x}$. Then $M$ necessarily has the form $\cl\cone F'$ where $F'$ is an exposed face of $Q$.
Clearly we have $\bar{x}\in F'$ and hence $F\subset F'$. The claim follows. Proof of $4$ is similar. Equivalence (\ref{eqn:equiv}) follows easily from the proofs of $1$ and $2$.\qed
\end{proof}

\begin{cor}[Minimal exposed faces of convex sets]\label{cor:rep_conv_set} \hfill \\
Consider a closed convex set $Q\subset\E$ and a point $\bar{x}$ in $Q$. Then for any vector $v\in \ri N_Q(\bar{x})$, the set $F=\partial \delta^{*}_Q(v)$
is a minimal exposed face at $\bar{x}$.
\end{cor}
\begin{proof}
Suppose without loss of generality $0\in\inter Q$ and let $\gamma_Q\colon \E\to\R$ be the gauge of $Q$. Fix a vector $v\in \ri N_Q(\bar{x})$ and observe by Proposition~\ref{prop:gauge}, we have $(v,-1)\in\ri N_{\sepi \gamma_Q}(\bar{x},1)$. 
It follows that $(N_{\sepi \gamma_Q})^{-1}(v,-1)=\cl\cone (F\times \{1\})$ is a minimal exposed face of $\epi \gamma_Q$ at $(\bar{x},-1)$ and consequently $F$ is a minimal exposed face of $Q$ at $\bar{x}$.\qed
\end{proof}

\begin{proof}{\em of Theorem~\ref{thm:epi_coh}}: Claim 1 is obvious. To see 2, let $F$ be a minimal exposed face of $f$ at $\bar{x}$. Then $\gph f\big|_F$ is an exposed face of $\epi f$. Choose a vector $(v,-\beta)\in \ri N_{\sepi f}(\bar{x},f(\bar{x}))$. Clearly then the inequality $\beta >0$ holds and we deduce using Corollary~\ref{cor:rep_conv_set} that the vector $(\frac{v}{\beta},-1)$  
exposes a minimal exposed face $M$ of $\epi f$ at $(\bar{x},f(\bar{x}))$. On the other hand, since $F$ is minimal exposed face of $f$ at $\bar{x}$, We immediately obtain the inclusion $\gph f\big|_F\subset M$, thereby establishing the validity of $2$. \qed
\end{proof}

\begin{proof}{\em of Corollary~\ref{thm:min_func_rep}}:
This is immediate from Corollary~\ref{cor:rep_conv_set} and Theorem~\ref{thm:epi_coh}.\qed
\end{proof}

\end{appendices}

%
%

\end{document}